\documentclass[12pt]{amsart}
\usepackage{amssymb}
\usepackage{color}
\usepackage{amsmath,epic,curves,amscd}
\usepackage[english]{babel}
\usepackage{graphicx}
\usepackage{comment}
\usepackage{appendix}
\usepackage{mathdots}
\usepackage[all]{xy}
\usepackage{mathtools}
\usepackage{lscape}
\usepackage{stmaryrd}
\usepackage{mathabx}
\newtheorem{claim}{}[section]
\newtheorem{theorem}[claim]{Theorem}

\newtheorem{lemma}[claim]{Lemma}
\newtheorem{proposition}[claim]{Proposition}
\newtheorem{corollary}[claim]{Corollary}
\theoremstyle{remark}
\newtheorem{remark}[claim]{Remark}
\theoremstyle{example}

\renewenvironment{proof}{\noindent{\it Proof. \hskip0pt}}
                      {$\square$\par\medskip}
\allowdisplaybreaks
\usepackage{color}

\begin{document}
\baselineskip 6.0 truemm
\parindent 1.5 true pc

\newcommand\lan{\langle}
\newcommand\ran{\rangle}
\newcommand\tr{\operatorname{Tr}}
\newcommand\ot{\otimes}
\newcommand\ttt{{\text{\sf t}}}
\newcommand\rank{\ {\text{\rm rank of}}\ }
\newcommand\choi{{\rm C}}
\newcommand\choid{{\rm D}}
\newcommand\dual{\star}
\newcommand\flip{\star}
\newcommand\cp{{{\mathbb C}{\mathbb P}}}
\newcommand\ccp{{{\mathbb C}{\mathbb C}{\mathbb P}}}
\newcommand\pos{{\mathcal P}}
\newcommand\tcone{T}
\newcommand\mcone{K}
\newcommand\superpos{{{\mathbb S\mathbb P}}}
\newcommand\blockpos{{{\mathcal B\mathcal P}}}
\newcommand\jc{{\text{\rm JC}}}
\newcommand\dec{{\mathbb D}{\mathbb E}{\mathbb C}}
\newcommand\decmat{{\mathcal D}{\mathcal E}{\mathcal C}}
\newcommand\ppt{{\mathcal P}{\mathcal P}{\mathcal T}}
\newcommand\pptmap{{\mathbb P}{\mathbb P}{\mathbb T}}
\newcommand\join{\vee}
\newcommand\meet{\wedge}
\newcommand\alg{{\text{\rm alg}}}
\newcommand\pr\prime
\newcommand\id{{\text{\rm id}}}
\newcommand\calm{{\mathcal M}}
\newcommand\calb{{\mathcal B}}
\newcommand\cala{{\mathcal A}}
\newcommand\cale{{\mathcal E}}
\newcommand\calf{{\mathcal F}}
\newcommand\calp{{\mathcal P}}
\newcommand\calt{{\mathcal T}}
\newcommand\cals{{\mathcal S}}
\newcommand\call{{\mathcal L}}
\newcommand\calr{{\mathcal R}}
\newcommand\caln{{\mathcal N}}
\newcommand\calcb{{\mathbb{CB}}}
\newcommand\calcp{{\mathbb{CP}}}
\newcommand\calnl{{\mathcal N\mathcal L}}
\newcommand\ad{{\text{\rm Ad}}}
\newcommand\ro{{\boxtimes}}
\newcommand\algmm{{\rm alg}(\calm^\pr,\calm)}
\newcommand\calhs{{\mathcal{HS}}}
\newcommand\rk{{\rm rank}\,}
\newcommand\e{\varepsilon}
\newcommand\conv{{\text{\rm conv}\,}}

\title{Infinite dimensional analogues of Choi matrices}

\author{Kyung Hoon Han, Seung-Hyeok Kye and Erling St\o rmer}
\address{Kyung Hoon Han, Department of Data Science, The University of Suwon, Gyeonggi-do 445-743, Korea}
\email{kyunghoon.han at gmail.com}
\address{Seung-Hyeok Kye, Department of Mathematics and Institute of Mathematics, Seoul National University, Seoul 151-742, Korea}
\email{kye at snu.ac.kr}
\address{Erling St\o rmer, Department of Mathematics, University of Oslo, 0316 Oslo, Norway}
\email{erlings at math.uio.no}

\thanks{The first and second author were partially supported by NRF-2020R1A2C1A01004587, Korea.}
\thanks{Parts of the contents in this paper were presented by the third author at
Online Workshop Celebrating the Choi-Jamio\l kowski Isomorphism, hosted at Torun, Poland, in March 2023.}

\begin{abstract}
For a class of linear maps on a von Neumann factor, we associate two objects, bounded operators and
trace class operators, both of which play the roles of Choi matrices. Each of them is positive if and only if the original map
on the factor is completely positive.
They are also useful to characterize positivity of maps as well as
complete positivity. It turns out that such correspondences are
possible for every normal completely bounded map if and only if the
factor is of type I. As an application, we provide criteria for
Schmidt numbers of normal positive functionals in terms of Choi
matrices of $k$-positive maps, in infinite dimensional cases. We
also define the notion of $k$-superpositive maps, which turns out to
be equivalent to the property of $k$-partially entanglement
breaking.
\end{abstract}
\maketitle

\section{Introduction}

Choi matrices of linear maps between matrix algebras, defined by M.-D.\ Choi \cite{choi75-10} in 1975,
have played fundamental roles in current quantum information theory since its beginning.
In fact, Woronowicz \cite{woronowicz} used Choi matrices in 1976,
to show that every positive linear map from $2\times 2$ matrices
into $n\times n$ matrices is decomposable if and only if $n\le 3$, which extends the result
of the third author \cite{stormer} for $n=2$. Woronowicz actually proved
the dual statements that there exists no $2\otimes 3$ entangled state
with positive partial transpose in the current terminologies,
and used a bilinear pairing between linear maps and block matrices through Choi matrices,
to get the required result on positive maps.

Choi's theorem tells us that a map is completely positive if and only if its Choi
matrix is positive (semi-definite). This Choi's theorem has been extended for infinite dimensional cases
by the third author \cite{stormer-dual}, to discuss the extension problem for a class of positive maps
arising from the theory of operator algebras. The second author \cite{eom-kye} also utilized
the bilinear pairing through Choi matrices to show that the dual objects of $k$-positive maps are convex sums
of rank one projections in block matrices whose range vectors correspond to matrices with rank
$\le k$, which is nothing but the notion of Schmidt numbers
\cite{terhal-schmidt} in the current terminologies.
The notion of Schmidt numbers also was defined in \cite{Shirokov_2013} for infinite dimensional case.
The notion of $k$-superpositive maps has been introduced more recently
\cite{{ando-04},{ssz}}, as completely positive maps whose Choi matrices have Schmidt numbers $\le k$.
Then $k$-superpositivity and $k$-positivity are dual objects by the bilinear pairing between linear maps
through Choi matrices. Choi matrices have been considered
for  multi-linear maps
\cite{{han_kye_tri},{han_kye_multi},{kye_multi_dual}},
as well as in various infinite dimensional cases
\cite{{bolanos-quezada},{duvenhage_snyman},{Friedland_2019},{grabowski_kus_marmo},{guddar},{Haapasalo_2020},{holevo_2011},{holevo_2011_a},{li_du},{Magajna_2021},{stormer_2015}}.

The Choi matrix of a linear map $\phi:M_n\to M_n$ is  the block matrix
$\choi_\phi = [\phi(e_{ij})]$ in $M_{n^2}$, with the usual matrix units $\{e_{ij}: i,j = 1,2,\dots,n\}$
in $M_n$.
When one writes $M_{n^2}$ as the tensor product $M_n \otimes M_n$ and
identify $M_n$ with $I \otimes  M_n$, then
$$
\choi_\phi = \sum  e_{ij} \otimes \phi(e_{ij})  = \id \otimes \phi(E_0),
$$
where $E_0$ is the range projection, up to scalar multiplications, of the separating and cyclic vector
$x_0 = \sum e_i \otimes e_i$ for $M_n$, where $\{e_1, e_2,\dots, e_n\}$  is the standard basis for $\mathbb C^n$.

In the present paper, we shall study generalizations of
$\choi_\phi$ to linear maps of von Neumann algebras
which are factors $\calm$ acting on a Hilbert space $H$,
and which have a separating and cyclic vector $x_0$.
Note that this is the case for the GNS representation for a faithful normal state.
Let $E_0$ be, as above, the projection onto the linear span of $x_0$.
Then the algebra $\algmm$ generated by $\calm$ and its commutant $\calm^\pr$
is weakly dense in $\calb(H)$, which is $*$-isomorphic to the von Neumann algebra tensor product
$\calm^\pr\overline\otimes\calm$
of  $\calm^\pr$  and  $\calm$ only when  $\calm$  is of type I.
We must then change the map ${\id}\otimes \phi$ above for $\phi:\calm\to\calm$ to
$\tilde\phi:\algmm\to\algmm$  by
$$
\tilde\phi(A^\pr B) = A^\pr\phi(B) =\phi(B)A^\pr,\qquad  A^\pr\in \calm^\pr,\ B \in \calm.
$$
When $\tilde\phi$ extends to a normal map, that is,
a map which is continuous with respect to the $\sigma$-weak operator topology on $\calb(H)$, we define
the bounded operator $\choi_\phi=\tilde\phi(E_0)\in\calb(H)$ and
the trace class operator $\choid_\phi=\tilde\phi_*(E_0)\in\calt(H)=\calb(H)_*$,
with the rank one projection $E_0$ onto $x_0$, where $\tilde\phi_*$ denotes the predual map of $\tilde\phi$.

We will see in Section \ref{sec_a} that both $\phi\mapsto \choi_\phi$ and
$\phi\mapsto \choid_\phi$ are injective with dense ranges, and provide
a general framework to get an equivalent condition for entanglement breaking property of $\phi$ in terms of $\choid_\phi$.
We show in Section \ref{sec_b} that both $\choi_\phi$ and $\choid_\phi$
enjoy the Choi's correspondence between positivity of them and
complete positivity of $\phi$ with Kraus decompositions, whenever $\tilde\phi$
extends to a normal map on the whole algebra $\calb(H)$.
We also characterize positivity of the map $\phi$ in terms of  $\choid_\phi$,
which are infinite dimensional analogues of Jamio\l kowski's theorem \cite{jam_72}.

We note that if $\tilde\phi$ extends to a normal map on $\calb(H)$ then $\phi$ itself is a normal completely bounded map.
We will see in Section \ref{sec_c} that the converse holds if and only if $\calm$ is of type ${\rm I}$,
and concentrate on the type ${\rm I}_\infty$ factor $\calm=\calb(K)$ acting on $H=\calhs(K)$ of Hilbert-Schmidt
operators by left multiplications.
In case of matrix algebras, all the possible variants of Choi matrices were found in \cite{{han_kye_Choi_mat},{kye_Choi_matrix}},
to retain the Choi's correspondences between complete positivity of maps and positivity of Choi matrices.
Choosing a variant among them turns out to amount to choosing a separating and cyclic vector in our approach.
See Remark \ref{rem-finite}.

When $\calm=\calb(K)$ is of type ${\rm I}$, we consider  in
Section 5 the ranks of Hilbert-Schmidt operators on
which $\calm$ acts by left multiplications. This leads us to get
infinite dimensional analogues for characterization of
$k$-positivity and Schmidt numbers in terms of Choi matrices. In
Section 6, we define the notion of $k$-superpositive maps as the
limit of a sequence of sums of the maps $X\mapsto V^*XV$ with $\rk
V\le k$ in terms of $\choid_\phi$, and show that they exhaust all
$k$-partially entanglement breaking maps.

In this note, we use the bilinear pairing between the space $\calb(H)$ of bounded operators and the space $\calt(H)$
of trace class operators, defined by
$\lan A,B\ran=\tr(A B)$ for $A\in \calb(H),\ B\in\calt(H)$,
where $\tr$ is defined by $\tr(X)=\sum_{i=1}^\infty (Xe_i|e_i)$ for an orthonormal basis $\{e_i\}$ of $H$,
and $(\, \cdot\, |\,\cdot\, )$ denotes the inner product which is linear in the first variable and
conjugate-linear in the second variable.
Then the predual map $\Phi_*:\calt(H)\to\calt(H)$ of a normal map $\Phi:\calb(H)\to\calb(H)$
is defined by
$\lan A,\Phi_*(B)\ran=\lan\Phi(A),B\ran$ for $A\in \calb(H)$ and $B\in\calt(H)$.
For given $x,y\in H$, we denote  by $x\ro y$ the rank one operator given by $x\ro y(z)=(z|y)x$.
By the relation $\lan A, x\ro y\ran= (Ax|y)$ for $A\in\calb (H)$, we see that the trace class operator
$x\ro y\in\calt(H)$ corresponds to the normal functional $\omega_{x,y}\in\calb(H)_*$ given by
$\lan A,\omega_{x,y}\ran=(Ax|y)$.

The authors are grateful to the referee for informing them of the paper \cite{duvenhage_snyman}.

\section{Choi matrices as bounded operators and trace class operators}\label{sec_a}

We recall that any von Neumann algebra $\calm$ acting on
a separable Hilbert space has a representation with a separating and cyclic vector,
and $\calm$ and its commutant $\calm^\pr$
are anti-$*$-isomorphic to each other in this representation \cite{tomita_takesaki}.
Throughout this paper, we suppose that $\calm$ is a von Neumann factor acting on a separable Hilbert space $H$
with a fixed separating and cyclic unit vector $x_0$.
For a given linear map $\phi:\calm\to\calm$, the map
\begin{equation}\label{tildephi}
\tilde\phi: \textstyle\sum_i A_i^\pr B_i\mapsto \textstyle\sum_i  A_i^\pr\phi(B_i),
     \qquad A_i^\pr\in \calm^\pr,\ B_i\in \calm
\end{equation}
is a well-defined linear map on  $\algmm$, because $\algmm$ is $*$-isomorphic to
the algebraic tensor product $\calm^\pr\ot\calm$ \cite{MvN}. See also \cite[Theorem 5.5.4]{KRI}.
We also note that $\algmm$ is weakly dense in $\calb(H)$ by the double commutant theorem of von Neumann.

\begin{proposition}\label{bimodule}
Suppose that $\Phi:\algmm\to\algmm$ is a linear map. Then
$\Phi=\tilde\phi$ for a linear map $\phi:\calm\to\calm$ if and only
if $\Phi$ is an $\calm^\pr$-bimodule map.
\end{proposition}

\begin{proof}
For every $A^\pr, B^\pr, D^\pr\in\calm^\pr$ and $C\in\calm$, we have
the identity
$$
\tilde\phi(A^\pr B^\pr C D^\pr)=\tilde\phi(A^\pr B^\pr D^\pr C)
=A^\pr B^\pr D^\pr\phi(C)=A^\pr B^\pr\phi(C) D^\pr= A^\pr
\tilde\phi( B^\pr C) D^\pr,
$$
and so we see that $\tilde\phi$ is an $\calm^\pr$-bimodule map. For
the converse, we suppose that $\Phi$ is an $\calm^\pr$-bimodule map.
If $A\in\calm$ and $E^\pr\in\calm^\pr$ then we have
$$
E^\pr \Phi(A)= \Phi(E^\pr A)=\Phi(A E^\pr)=\Phi(A) E^\pr,
$$
and so we see that $\Phi(A)\in\calm$ by the double commutant
theorem. Therefore, we see that the restriction $\phi=\Phi|_\calm$
maps $\calm$ into $\calm$. Now, we have $\tilde\phi(E^\pr A)=E^\pr
\phi(A) =E^\pr \Phi(A)=\Phi(E^\pr A)$, and so we conclude that
$\Phi=\tilde\phi$ on $\algmm$.
\end{proof}

With the $*$-isomorphism from $\algmm$ onto $\calm^\pr\ot\calm$, we
have the commuting diagram:
\begin{equation}\label{commdia}
\begin{CD}
\algmm @         >\tilde\phi>>      \algmm \\
      @V \simeq VV                   @VV \simeq V \\
\calm^\pr\ot\calm  @         >\id\ot\phi>>       \calm^\pr\ot\calm
\end{CD}
\end{equation}
We also note that the two $*$-algebras $\algmm$ and $\calm^\pr\ot\calm$
act on $H$ and $H\ot H$, respectively, and so we may consider the
$C^*$-algebra $C^*(\calm^\pr,\calm)\subset \calb(H)$ together with
the $C^*$-algebra  $\calm^\pr\ot_{\rm min}\calm\subset \calb(H\ot
H)$ generated by them. It turns out that $C^*(\calm^\pr,\calm)$ is $*$-isomorphic to the binormal tensor
product $\calm^\pr\ot_{\rm bin}\calm$ \cite{effros_lance}, and the isomorphism in
(\ref{commdia}) extends to a surjective $*$-homomorphism $\pi$ from
$\calm^\pr\ot_{\rm bin}\calm$ onto $\calm^\pr\ot_{\rm min}\calm$.
Note that $\pi$ is an isometry if and only if $\calm$ is injective.
See \cite{effros_lance} and \cite[Proposition 2.3.6]{blackadar}.
With the von Neumann algebra $\calm^\pr\overline\ot\calm$ acting on $H\ot H$,
we summarize our discussion in the following diagram:
$$
\xymatrix{
{\rm alg}(\calm^\pr,\calm) \ar[d]^\simeq \ar@{}|-*[@]{\subset}[r]
  & \calm^\pr\ot_{\rm bin}\calm = C^*(\calm^\pr,\calm) \ar[d]^\pi \ar@{}|-*[@]{\subset}[r] & \calb(H)\\
\calm^\pr\ot \calm \ar@{}|-*[@]{\subset}[r]
  & \calm^\pr\ot_{\rm min}\calm \ar@{}|-*[@]{\subset}[r] &  \calm^\pr\overline\ot\calm
}
$$

We note that if $\phi:\calm\to\calm$ is completely bounded then
$\id\ot\phi$ from $\calm^\pr\ot_{\rm min} \calm$ into itself is a bounded map.
The converse is also true, because every infinite dimensional factor
has a $*$-subalgebra isomorphic to a matrix algebra with arbitrary
size. See \cite[Proposition III.1.5.14]{blackadar}.
We also note that $\|\id_{\calm'}\ot\phi\|= \|\phi\|_{\rm cb}$.
When $\tilde\phi$ extends to a normal map on $\calb(H)$, we define $\choi_\phi$
and $\choid_\phi$ by
\begin{equation}\label{deff}
\choi_\phi=\tilde\phi(E_0)\in \calb(H),\qquad \choid_\phi=\tilde\phi_*(E_0)\in\calt(H),
\end{equation}
with the one dimensional projections $E_0=x_0\ro x_0$ onto the separating and cyclic unit vector $x_0$.
By the identification $x_0\ro x_0=\omega_{x_0,x_0}$ in $\calt(H)=\calb(H)_*$, we have
$\choid_\phi=\omega_{x_0,x_0}\circ\tilde\phi$ as a normal functional on $\calb(H)$.
In fact, we have $\lan X, \choid_\phi\ran=\lan \tilde\phi(X), E_0\ran=\omega_{x_0,x_0}(\tilde\phi(X))$ for $X\in\calb(H)$.

We denote by $\calcb^\sigma(\calm)$ the space of all completely bounded linear maps $\phi:\calm\to\calm$
which are normal. We also denote by $\calcb^\sigma_\e(\calm)$ the space of maps $\phi:\calm\to\calm$
such that $\tilde\phi$ extends to a normal map on $\calb(H)$. We note that $\calcb^\sigma_\e(\calm)\subset \calcb^\sigma(\calm)$
by \cite[Proposition III.1.5.14]{blackadar} again.
Elementary operators $\phi_{M,N}$ given by
$$
\phi_{M,N}(X)=MXN,\qquad M,N\in\calm
$$
are typical examples of maps belonging to $\calcb^\sigma_\e(\calm)$.
Especially, the map $\ad_V$ is defined by $\ad_V(A)=V^*AV$ for $A\in\calm$.
We have the identity
\begin{equation}\label{choi_ele}
\choi_{\phi_{M,N}}=ME_0N=Mx_0\ro N^*x_0\in\calb(H).
\end{equation}
By the identity
$$
\lan A^\pr B, \choid_{\phi_{M,N}} \ran
=(A^\pr\phi_{M,N}(B)x_0|x_0)
=(MA^\pr BN x_0|x_0)=(A^\pr BN x_0|M^*x_0),
$$
for $A^\pr\in\calm^\pr$ and $B\in\calm$, we have
\begin{equation}\label{choid_ele}
\choid_{\phi_{M,N}}=\omega_{Nx_0,M^*x_0} =Nx_0\ro M^*x_0 =NE_0M\in\calt(H).
\end{equation}
These identities (\ref{choi_ele}) and (\ref{choid_ele}) have been used as the definitions in \cite{stormer_2015}
for elementary operators, which are the main motivations for our definitions in (\ref{deff}).

\begin{theorem}\label{inj}
Suppose that $\calm$ is a factor acting on a separable Hilbert space $H$
with a separating and cyclic vector $x_0$. Then both
correspondences $\phi\mapsto \choi_\phi\in\calb(H)$ and
$\phi\mapsto\choid_\phi\in\calt(H)$ are injective on the space
$\calcb^\sigma_\e(\calm)$ with the weak dense and norm dense ranges, respectively.
\end{theorem}

\begin{proof}
For every $V,{W}\in\calm$, we note that the rank one {operator $V x_0\ro Wx_0$}
belongs to the ranges of both correspondences. Because $x_0$ is
cyclic for $\calm$, we see that the $\choi_\phi$'s and $\choid_\phi$'s make a weakly dense subspace of $\calb(H)$,
and a norm dense subspace of $\calt(H)$, respectively.
By the $\calm^\pr$-bimodule property
$$
\tilde\phi(A^{\pr *}E_0B^\pr)= A^{\pr *}\tilde\phi(E_0)B^\pr= A^{\pr *}\choi_\phi B^\pr,\qquad A^\pr, B^\pr\in\calm^\pr,
$$
we see that $\choi_\phi=0$ implies $\tilde\phi=0$,
because $x_0$ is cyclic for $\calm^\pr$. Therefore, we have $\phi=0$. We also have
$$
(\phi(A)B^\pr x_0|C^\pr x_0) =(C^{\pr *}\phi(A)B^\pr x_0| x_0)
=(\tilde\phi(C^{\pr *}AB^\pr) x_0|x_0)= \lan C^{\pr *}AB^\pr, \choid_\phi \ran,
$$
for every $A\in\calm$ and $B^\pr,C^\pr\in\calm^\pr$. Since $x_0$
is cyclic for $\calm^\pr$, we see that $\choid_\phi=0$ implies
$\phi=0$.
\end{proof}

\begin{theorem}\label{entbreak}
Suppose that $\calm$ is a factor acting on a separable Hilbert space $H$
with a separating and cyclic vector $x_0$, and a closed subset $S$ of $\calt(H)$ has the property that
$V^\pr\varrho V^{\pr *}\in S$ for every $V^\pr\in\calm^\pr$ and
$\varrho\in S$. For a normal completely bounded map $\phi\in \calcb^\sigma_\e(\calm)$,
the following are equivalent:
\begin{enumerate}
\item[{\rm (i)}]
$\choid_\phi\in S$;
\item[{\rm (ii)}]
$\tilde\phi_*(\omega_{x,x})\in S$ for every $x\in H$;
\item[{\rm (iii)}]
$\tilde\phi_*$ sends $\calt(H)^+$ into $S$.
\end{enumerate}
\end{theorem}

\begin{proof}
We note that $\tilde\phi_*$ enjoys the $\calm^\pr$-bimodule map property by Proposition \ref{bimodule}, and so
we see that (i) implies that $\tilde\phi_*(\omega_{V^\pr x_0,V^\pr x_0})=V^\pr\tilde\phi_*(E_0)V^{\pr *}\in S$ for $V^\pr\in\calm^\pr$.
Since $x_0$ is cyclic for $\calm^\pr$, we see that (i) implies (ii). We have the direction (ii) $\Longrightarrow$ (iii),
because $\tilde\phi_*$ is continuous and every positive $\varrho\in \calt(H)^+$ is the norm limit of sums of $\omega_{x,x}$.
\end{proof}

If $S$ is the set of all separable states (see Section 5) then the condition (iii) of Theorem \ref{entbreak} tells us that
$\tilde\phi_*$ sends every normal state to a separable state, that is, $\phi$ is entanglement breaking \cite{hsrus}.
Therefore, Theorem \ref{entbreak} tells us that $\phi\in\calcb^\sigma_\e(\calm)$ is entanglement breaking if and only if
$\choid_\phi$ is separable, as in cases of matrix algebras \cite{{hsrus},{ssz}}.

\section{Positivity and complete positivity}\label{sec_b}

Several authors \cite{bolanos-quezada, Friedland_2019, grabowski_kus_marmo, holevo_2011, holevo_2011_a, li_du, stormer-dual} associated the objects,
which may be considered as Choi matrices, for linear maps in various infinite
dimensional situations, and gave correspondences between complete
positivity of maps and positivity of the objects, to extend Choi's theorem for matrices \cite{choi75-10}. Especially, we
note that the functional $\omega_{x_0,x_0}\circ\tilde\phi$, which is
just $\choid_\phi$ in our approach, has been also considered in \cite{duvenhage_snyman, Haapasalo_2020} for
the correspondences between channels and bi-partite states for von Neumann algebras.

Now, we proceed to show that $\phi\in\calcb^\sigma_\e(\calm)$ is
completely positive if and only if $\choi_\phi$ is positive if and
only if $\choid_\phi$ is positive.
The following was shown by the third author \cite{stormer_2015}
for elementary operators of the form $X\mapsto \sum_{i=1}^n M_i XN_i $ on injective factors.

\begin{theorem}\label{thm}
Suppose that $\calm$ is a factor acting on a separable Hilbert space
$H$ with a separating and cyclic vector $x_0$. For a linear map
$\phi\in\calcb^\sigma_\e(\calm)$, the following are equivalent:
\begin{enumerate}
\item[{\rm (i)}]
$\phi:\calm\to\calm$ is a completely positive linear map;
\item[{\rm (ii)}]
$\tilde\phi:\calb(H)  \to \calb(H)$ is a completely positive linear map;
\item[{\rm (iii)}]
$\tilde\phi:\calb(H)  \to \calb(H)$ is a positive linear map;
\item[{\rm (iv)}]
$\choi_\phi$ is a positive operator in $\calb(H)$;
\item[{\rm (v)}]
$\choid_\phi$ is a positive operator in $\calt(H)$;
\item[{\rm (vi)}]
$\phi=\sum_{i=1}^\infty\ad_{V_i}$ in the point--weak topology
with $V_i\in\calm$ and $\sum_{i=1}^\infty{V_i^*V_i}\le \|\phi\|_{\rm cb}I$.
\end{enumerate}
\end{theorem}

\begin{proof}
We first show that (i), (ii) and (iii) are equivalent.
Suppose that $\phi$ is completely positive. Then the map
${\rm id} \otimes \phi : \mathcal M^\pr \otimes_{\rm bin} \mathcal M \to \mathcal M^\pr \otimes_{\rm bin} \mathcal M$
is completely positive by
\cite[Proposition IV.2.3.4(ii)]{blackadar}, and so we see that $\tilde\phi$ is completely positive on
$C^*(\calm^\pr,\calm)=\calm^\pr\ot_{\rm bin}\calm$, which is weakly dense in $\calb(H)$.
Since $\tilde\phi$ is normal on $\calb(H)$, we conclude
that $\tilde \phi : \calb(H) \to \calb(H)$ is also completely positive.
The direction (ii) $\Longrightarrow$ (iii) is clear. For the direction (iii) $\Longrightarrow$ (i),
we may assume that $\calm$ is not finite dimensional. For every $n=1,2,\dots$, the von Neumann algebra $\calm^\pr$
contains a $*$-subalgebra which is isomorphic to $M_n$ by \cite[Proposition III.1.5.14]{blackadar}.
Since $M_n$ is nuclear, the norm inherited from the inclusion
$$
M_n \otimes \calm \subset \calm^\pr \otimes_{\rm bin} \calm \subset \calb(H)
$$
is the unique $C^*$-tensor norm. Therefore, the restriction of the positive map $\tilde\phi:\calb(H)\to\calb(H)$
on $M_n \otimes \calm$ is also positive, and we see that $\phi$ is completely positive.

Now, we show that the statements (iii), (iv) and (v) are also equivalent. The direction (iii) $\Longrightarrow$ (iv)
is clear since $E_0$ is positive. For every $A^\pr\in\calm^\pr$, we have
$$
\lan A^{\pr *} E_0 A^\pr, \choid_\phi \ran
=\lan \tilde\phi(A^{\pr *} E_0 A^\pr), E_0 \ran
=\lan A^{\pr *} \tilde\phi(E_0) A^\pr, E_0 \ran
=\lan \choi_\phi, A^\pr E_0 A^{\pr *}\ran.
$$
Because $x_0$ is cyclic for $\calm^\pr$, we see that $\choid_\phi$ is
positive if and only if $\choi_\phi$ is  positive.
For every $X\in \calb(H)$ and $B^\pr\in\calm^\pr$, we also have
$$
(\tilde\phi(X)B^\pr x_0|B^\pr x_0)
=(B^{\pr *}\tilde\phi(X)B^\pr x_0|x_0)
=(\tilde\phi (B^{\pr *}XB^\pr) x_0|x_0)
=\lan B^{\pr *}XB^\pr, \choid_\phi\ran,
$$
by the $\calm^\pr$-bimodule map property of $\tilde\phi$.
Suppose that  $\choid_\phi$ is positive. Then we have $(\tilde\phi(X)B^\pr x_0|B^\pr x_0)\ge 0$
for every  $B^\pr\in\calm^\pr$ and positive $X\in\calb(H)$.
Since $x_0$ is cyclic for $\calm^\pr$, we see that $\tilde\phi(X)$ is positive whenever
$X\in\calb(H)$ is positive.

It remains to show that (vi) is also equivalent to the other statements. It is clear that (vi) implies that
$\phi$ is completely positive. We will show that (ii) implies (vi).
To do this, we use Kraus' theorem \cite{kraus-71} that a normal completely positive map $\tilde\phi$
from $\calb(H)$ into itself is of the form $\tilde\phi(X)=\sum_{i=1}^\infty V_i^*XV_i$ with $V_i\in \calb(H)$
and $\sum_{i=1}^\infty{V_i^*V_i}\le \|\phi\|_{\rm cb}I$.
Then it remains to show that $V_i\in \calm$.
For every $A^\pr \in \mathcal M'$, we have
$$
\begin{aligned}
\sum_{i=1}^\infty (A^\pr V_i - &V_i A^\pr)^*(A^\pr V_i - V_i A^\pr)\\
& = \sum_{i=1}^\infty V_i^*A^{\pr *}A^\pr V_i - \sum_{i=1}^\infty V_i^*A^{\pr *} V_i A^\pr
    - \sum_{i=1}^\infty A^{\pr *} V_i^* A^\pr V_i + \sum_{i=1}^\infty A^{\pr *} V_i^*V_iA^\pr \\
& = \tilde \phi(A^{\pr *} A^\pr) - \tilde \phi(A^{\pr *})A^\pr - A^{\pr *}\tilde \phi(A^\pr) + A^{\pr *}\tilde \phi(I)A^\pr \\
& = 0,
\end{aligned}
$$
by the $\mathcal M'$-bimodule map property of $\tilde \phi$.
This implies that $A^\pr V_i = V_i A^\pr $ for all $A^\pr \in \mathcal M'$.
By the double commutant theorem, each $V_i$ belongs to $\mathcal M$.
\end{proof}

\begin{corollary}\label{cor}
For a unit vector $x\in H$ and the projection $E$ onto the subspace
generated by $x$, the following are equivalent:
\begin{enumerate}
\item[{\rm (i)}]
there exists $\phi\in\calcb^\sigma_\e(\calm)$ such that
$\choi_\phi=E$;
\item[{\rm (ii)}]
there exists $\phi\in\calcb^\sigma_\e(\calm)$ such that
$\choid_\phi=E$;
\item[{\rm (iii)}]
there exists $A\in\calm$ such that $A x_0=x$.
\end{enumerate}
\end{corollary}

\begin{proof}
For the direction (iii) $\Longrightarrow$ (i) and (iii)
$\Longrightarrow$ (ii), we may take $\phi=\ad_{A^*}$ and
$\phi=\ad_{A}$, respectively. Now, we suppose that $\choi_\phi=E$ with
$\phi\in\calcb^\sigma_\e(\calm)$. Then $\phi(X)=\sum_{i=1}^\infty
V_i^* X V_i$ with $V_i\in\calm$ and $\sum_{i=1}^\infty{V_i^*V_i}\le
\|\phi\|_{\rm cb}I$ by Theorem \ref{thm}. Then, we have
$$
V_i^*E_0V_i\le \sum_{i=1}^\infty V_i^*E_0V_i=\tilde\phi(E_0)=E,
$$
for every $i=1,2,\dots$, and we see that $\lambda_i V_i^*x_0= x$ for some $i$
with a nonzero scalar $\lambda_i$. We take $A=\lambda_i V_i^*$, to get $A x_0=x$. Therefore, we proved that (i) implies (iii).
Exactly the same argument works for the direction
(ii) $\Longrightarrow$ (iii).
\end{proof}

Corollary \ref{cor} tells us that $\{Ax_0: A\in\calm\}=H$ is a necessary condition for the
surjectivity of the correspondences $\phi\mapsto \choi_\phi$ and $\phi\mapsto \choid_\phi$.
We show that it is also sufficient.

\begin{theorem}\label{surj}
Suppose that $\calm$ is a factor acting on a separable Hilbert space
with a separating and cyclic vector $x_0$. Then the following are
equivalent:
\begin{enumerate}
\item[{\rm (i)}]
$\phi\mapsto\choi_\phi:\calcb^\sigma_\e(\calm)\to\calb(H)$ is
surjective;
\item[{\rm (ii)}]
$\phi\mapsto\choid_\phi:\calcb^\sigma_\e(\calm)\to\calt(H)$ is
surjective;
\item[{\rm (iii)}]
$\{Ax_0: A\in\calm\}=H$;
\item[{\rm (iv)}]
$\calm$ is finite dimensional.
\end{enumerate}
\end{theorem}

\begin{proof}
We  note that (i) implies (iii), and (ii) implies (iii) by Corollary \ref{cor}.
Because (iv) also implies (i) and (ii), it remains to show that (iii) implies (iv).
Suppose that $\mathcal M {x}_0 = H$.
Since ${x}_0$ is a separating vector for $\mathcal M$ and $\|Ax_0\|_H \le \|A\|_{\mathcal M} \|{x}_0\|_H$, the linear map
$$
\Phi : A \in \mathcal M \mapsto A {x}_0 \in H
$$
is a continuous bijection.
By the open mapping theorem, $\Phi$ is a homeomorphism actually.
If $\mathcal M$ is infinite dimensional, then it contains
the $n$-dimensional $L^\infty$-space $\ell^\infty_n$ for each $n \in \mathbb N$,
and the subspace $\Phi(\ell^\infty_n)$ of $H$ is isometric to $\ell^2_n$,
as an $n$-dimensional subspace of the Hilberst space $H$.
We have
$$
\|\Phi|_{\ell^\infty_n}\|~ \|\Phi^{-1}|_{\ell^2_n}\| \le \|\Phi\|~ \|\Phi^{-1}\| < \infty,
$$
for each $n=1,2,\dots$.
However, the Banch-Mazur distance between $\ell^\infty_n$ and $\ell^2_n$ is $\sqrt{n}$
\cite[Proposition 37.6]{tomczak}, which goes to infinity as $n \to \infty$.
\end{proof}

We present one more application of Theorem \ref{thm} in relation to separability criteria.
For a given $V\in\calm$, we recall that the map $\ad_V$ is defined by $\ad_V(A)=V^*AV$ for $A\in\calm$.

\begin{theorem}\label{horo}
Suppose that $\calm$ is a factor acting on a separable Hilbert space
with a separating and cyclic vector $x_0$,
and $S$ is a convex cone in $\calcb^\sigma_\e(\calm)$ satisfying the condition that
$\psi\circ\ad_V$ belongs to $S$ for every $V\in\calm$ and $\psi\in S$.
For $\phi\in\calcb^\sigma_\e(\calm)$, the
following are equivalent:
\begin{enumerate}
\item[{\rm (i)}]
$\lan \choi_\psi, \choid_\phi \ran \ge 0$ for every $\psi\in S$;
\item[{\rm (ii)}]
${\phi\circ\psi}$ is completely positive for every $\psi\in S$;
\item[{\rm (iii)}]
${\tilde\phi(\choi_\psi)}$ is positive for every $\psi\in S$;
\item[{\rm (iv)}]
${\tilde\psi_*(\choid_\phi)}$ is positive for every $\psi\in S$.
\end{enumerate}
\end{theorem}

\begin{proof}
We first note that (i) holds if and only if $\lan \choi_{\psi\circ\ad_V}, \choid_\phi \ran \ge 0$
for every $\psi\in S$ and $V\in\calm$.
We have
$$
\begin{aligned}
\lan \choi_{\psi \circ \ad_V}, \choid_\phi \ran
&= \lan \tilde \psi (V^* E_0 V), \tilde \phi_* (E_0) \ran\\
&= \lan V^* E_0 V, \tilde \psi_* \circ \tilde \phi_* (E_0) \ran\\
&= \lan V^* E_0 V, \choid_{\phi \circ \psi} \ran
= ( \choid_{\phi \circ \psi} V^*x_0 | V^*x_0),
\end{aligned}
$$
for $\psi \in S$ and $V \in \mathcal M$.
Since $x_0$ is cyclic for $\calm$,
we see that (i) holds if and only if $\choid_{\phi \circ \psi}$ is positive if and only if
${\phi\circ\psi}$ is completely positive
by Theorem \ref{thm}.
The equivalences (ii) $\Longleftrightarrow$ (iii) and (ii) $\Longleftrightarrow$ (iv) follow from the identities
{$\tilde\phi(\choi_\psi)=\choi_{\phi\circ\psi}$} and {$\tilde\psi_*(\choid_\phi)=\choid_{\phi\circ\psi}$},
respectively, together with Theorem \ref{thm} again.
\end{proof}

In case of matrix algebras, we note that the statement (iv) in Theorem \ref{horo}
is the Horodecki's separability criteria \cite{horo-1} for
$\choid_\phi$ when $S$ is the set of all positive maps.
Furthermore, Theorem \ref{horo} (ii) is one of equivalent conditions for entanglement breaking channels in \cite{hsrus}.
The condition that
$\psi\circ\ad_V$ belongs to $S$ for every $V\in\calm$ and
$\psi\in S$ gives rise to the notion of
right mapping cones \cite{gks} in case of matrix algebras.

A variant $\sum_{i,j=1}^n e_{j,i}\ot \phi(e_{i,j})$ of the Choi
matrix also had been considered by de Pillis \cite{dePillis} and
Jamio\l kowski \cite{jam_72} prior to Choi \cite{choi75-10}, to
characterize Hermiticity preserving maps and positive maps between
matrix algebras. See also
\cite{{Paulsen_Shultz},{kye_Choi_matrix},{han_kye_Choi_mat}} for
further variants. It is easy to characterize positivity of $\phi$ in
terms of $\choid_\phi$. In fact, we do not need
$\phi\in\calcb^\sigma_\e(\calm)$ for this purpose. Recall that
$\choid_\phi=\omega_{x_0,x_0}\circ\tilde\phi$ when
$\phi\in\calcb^\sigma_\e(\calm)$, in the following:

\begin{theorem}\label{pos}
Suppose that $\calm$ is a factor acting on a separable Hilbert space with a separating and cyclic vector $x_0$.
For a linear map $\phi:\calm\to\calm$, the following are equivalent:
\begin{enumerate}
\item[{\rm (i)}]
$\phi$ is positive;
\item[{\rm (ii)}]
$(\omega_{x_0,x_0}\circ\tilde\phi)(A^\pr B)\ge 0$ for every $A^\pr\in \calm^{\pr +}$ and $B\in\calm^{+}$.
\end{enumerate}
\end{theorem}

\begin{proof}
Suppose that $\phi$ is positive. For positive $A^\pr\in\calm^{\pr +}$,
we write $A^\pr=C^{\pr *}C^\pr$ with $C^\pr\in \calm^{\pr}$. Then we have
\begin{equation}\label{cccc}
(\omega_{x_0,x_0}\circ\tilde\phi)(A^\pr B)=(\phi(B)C^\pr x_0|C^\pr x_0),
\end{equation}
which is nonnegative whenever $B$ is positive.
For the converse, we suppose that (ii) holds. Then the relation (\ref{cccc}) again
shows that $\phi(B)$ is positive whenever $B$ is positive, because
$x_0$ is cyclic for $\calm^\pr$.
\end{proof}

\section{Type ${\rm I}$ factors}\label{sec_c}

In this section, we will see that
$\calcb^\sigma_\e(\calm)= \calcb^\sigma(\calm)$ holds if and only if $\calm$ is of type ${\rm I}$.
We also characterize positivity of $\phi$ in terms of $\choi_\phi$ in case of type ${\rm I}$ factor.
For a given normal state $\sigma$ on $\calm$,
we consider the linear map $\phi:\calm\to\calm$ defined by
$$
\phi: A\mapsto \sigma(A)I,\qquad A\in\calm,
$$
which is unital completely positive and normal.
If $\tilde\phi$ extends to $\calb(H)$
then the extension $\tilde\phi:\calb(H)\to\calm^\pr$ is a normal conditional expectation by Proposition \ref{bimodule} and Theorem \ref{thm} (ii), and so
we see that $\calm^\pr$ must be of type I,
by Tomiyama's result \cite{tom59} (see also \cite[Theorem IV 2.2.2]{blackadar})
that if $\Phi:\calm_1\to \calm_2$
is a normal conditional expectation and $\calm_1$ is of type I then $\calm_2$ is also of type I.
Therefore, we see that if $\calcb^\sigma_\e(\calm)= \calcb^\sigma(\calm)$ then $\calm$ is of type ${\rm I}$.

Now, we consider the type ${\rm I}$ factor $\calm=\calb(K)$ with a separable Hilbert space $K$,
 and the Hilbert space $H={\mathcal H\mathcal S}(K)$ consisting of
Hilbert-Schmidt operators with the inner product
$$
(x|y)=\tr(xy^*)=\sum_{i,j}(xe_i|e_j)(e_j|ye_i),\qquad x,y\in H,
$$
for a given orthonormal basis $\{e_i\}$ of $K$.
We have
\begin{equation}\label{id-1}
(\xi_1\ro\eta_1)(\xi_2\ro\eta_2)
=(\xi_2|\eta_1)\xi_1\ro\eta_2,\qquad
(\xi\ro\eta)^*=\eta\ro\xi,\qquad
x(\xi\ro\eta)y=x\xi \ro y^*\eta,
\end{equation}
as operators on $K$, for $x,y\in \calb(K)$. We also have
\begin{equation}\label{id-2}
(\xi_1\ro\eta_1 | \xi_2\ro\eta_2)
=(\xi_1|\xi_2)(\eta_2|\eta_1),\qquad
(x| \xi\ro\eta)
=(x\eta|\xi),
\end{equation}
as vectors in $H={\mathcal H\mathcal S}(K)$,
for $x\in {\mathcal H\mathcal S}(K)$.

Suppose that $\calb(K)$ acts on $H$ by the left multiplication $L_x(y)=xy$ for $x\in\calb(K)$.
With the identification $H\simeq \bar K\ot  K$ by $\xi\ro\eta \leftrightarrow \bar\eta\ot\xi$,
we get the following commuting diagram:
\begin{equation}\label{tensor-left}
\begin{CD}
H @         >\simeq>>       \bar K\ot  K \\
      @V R_{y^*}L_x VV            @VV \bar y\ot x  V \\
H @         >\simeq>>       \bar K\ot K
\end{CD}
\end{equation}
for $x,y\in \calb(K)$, because $R_{y^*}L_x(\xi\ro\eta)=x(\xi\ro\eta)y^*=x\xi\ro y\eta\in H$
corresponds to $\overline{y\eta} \ot x\xi =(\bar y\ot x)(\bar\eta\ot\xi)\in \bar K\ot K$.
Then the von Neumann algebra
$$
L_{\calb(K)}=\{L_x\in\calb(H): x\in\calb(K)\}
$$
acting on $H$ has the commutant $R_{\calb(K)}=\{R_x:x\in\calb(K)\}$
with the action of $\calb(K)$ by right multiplications.
The von Neumann algebra generated by $L_{\calb(K)}$ and its commutant is the full algebra $\calb(H)$,
and we also have $\calb(\bar K)\overline\ot \calb(K)=\calb (\bar K\ot K)$.

Suppose that $\phi : \calb(K) \to \calb(K)$ is a normal completely bounded map.
It has the completely bounded predual map $\phi_* : \calt(K) \to \calt(K)$,
and the tensor product
$$
{\rm id}_{\calt(\bar K)}\otimes \phi_*  : \calt(\bar K) \widehat{\otimes} \calt(K) \to \calt(\bar K) \widehat{\otimes} \calt(K)
$$
through the operator space projective tensor product is also completely bounded.
We take the dual map
$$
 {\rm id}_{\calb(\bar K)} \otimes\phi  : \calb(\bar K) \widebar{\otimes} \calb(K) \to \calb(\bar K) \widebar{\otimes} \calb(K),
$$
which is a normal completely bounded map \cite{effros-ruan}.
Considering the commuting diagram
$$
\begin{CD}
\calb(H) @         >\simeq>>       \calb(\bar K\ot K) \\
      @V \tilde\phi VV                   @VV \id\ot \phi V \\
\calb(H) @         >\simeq>>       \calb(\bar K\ot K)
\end{CD}
$$
where $\tilde\phi$ is given by
$\tilde\phi( R_yL_x)=R_yL_{\phi(x)}$ for $x,y\in \calb(K)$, we conclude that
$\tilde\phi$ extends to a normal completely bounded map. We summarize as follows:

\begin{theorem}\label{extension}
For a factor $\calm$, the following are equivalent:
\begin{enumerate}
\item[{\rm (i)}]
the map $\tilde\phi$ extends to a normal completely bounded map on $\calb(H)$ for every $\phi\in \calcb^\sigma(\calm)$;
\item[{\rm (ii)}]
the map $\tilde\phi$ extends to a normal map on $\calb(H)$ for every $\phi\in \calcb^\sigma(\calm)$;
\item[{\rm (iii)}]
$\calm$ is of type {\rm I}.
\end{enumerate}
\end{theorem}

The direction (iii) $\Longrightarrow$ (i) of Theorem \ref{extension} can be also seen by a result \cite{hou_2010}
that every normal completely bounded map on a type ${\rm I}$ factor is the sum of elementary operators.
The following simple proposition will be useful to discuss a type ${\rm I}$ factor $\calb(K)$
which acts on $H=\calhs(K)$ by left multiplications.

\begin{proposition}\label{dense-vec}
If $x_0\in H$ is separating and cyclic  for $\calb(K)$ then we have the following:
\begin{enumerate}
\item[{\rm (i)}]
both $\{x_0\xi:\xi\in K\}$ and $\{x_0^*\xi:\xi\in K\}$ are dense in $K$;
\item[{\rm (ii)}]
for each $k=1,2,\dots$, the set $\{Vx_0\in \calhs(K) : V\in\calb(K),\ \rk V\le k\}$ is dense in $\{x\in\calhs(K):\rk x\le k\}$.
\end{enumerate}
\end{proposition}

\begin{proof}
We note that $x_0\in H$ is a separating and cyclic vector if and only if it is injective with a dense range
as an operator in $\calhs(K)$, and so the statement (i) follows.
In order to prove (ii), suppose that $x=\sum_{i=1}^k \xi_i\ro\eta_i\in\calhs(K)$.
We take $\zeta_i\in K$ such that $\|\eta_i-x_0^*\zeta_i\|<\e/M$ for $i=1,2,\dots,k$, with $M=\sum_{i=1}^k\|\xi_i\|$.
Putting $V=\sum_{i=1}^k\xi_i\ro\zeta_i\in\calb(K)$, we have
$$
\|x-Vx_0\|=
\|\textstyle\sum_{i=1}^k\xi_i\ro\eta_i-(\xi_i\ro \zeta_i) x_0\|=
\|\textstyle\sum_{i=1}^k\xi_i\ro (\eta_i-x_0^*\zeta_i)\|<\e,
$$
as was required.
\end{proof}

For given normal functionals $\sigma^\pr\in\calm^\pr_*$ and $\tau\in\calm_*$, we define
the linear functional $\sigma^\pr\cdot\tau$ on $\algmm$ by
$$
\lan A^\pr B, \sigma^\pr\cdot\tau \ran
=\lan A^\pr, \sigma^\pr \ran \lan B, \tau \ran,\qquad A^\pr\in\calm^\pr,\ B\in\calm.
$$
Suppose that $\calm=\calb(K)$ as before.
By the identification $\calb(H)\simeq\calb(\bar K\ot K)$, we also see that
$\sigma^\pr \cdot\tau$ extends to a normal functional, and
$\sigma^\pr\cdot\tau$ is positive whenever $\sigma^\pr$ and $\tau$ are positive.
For a given $\phi\in\calcb^\sigma(\calb(K))$, we have
$$
\begin{aligned}
\lan A^\pr B, \tilde\phi_*(\sigma^\pr\cdot\tau) \ran
&=\lan A^\pr \phi(B), \sigma^\pr\cdot\tau \ran\\
&=\lan A^\pr, \sigma^\pr \ran \lan \phi(B), \tau \ran\\
&=\lan A^\pr, \sigma^\pr \ran \lan B, \phi_*(\tau) \ran
=\lan A^\pr B, \sigma^\pr\cdot\phi_*(\tau) \ran,
\end{aligned}
$$
for $A^\pr\in\calm^\pr$ and $B\in\calm$. Therefore, we have the identity
\begin{equation}\label{iddddxxx}
\tilde\phi_*(\sigma^\pr\cdot\tau)=\sigma^\pr\cdot\phi_*(\tau),
\qquad \sigma^\pr\in\calm^\pr_*,\ \tau\in\calm_*.
\end{equation}
This identity may be considered as the dual object of the identity (\ref{tildephi}), which
has been used as the definition of $\tilde\phi$.

In order to characterize the positivity of $\phi$ in terms of $\choi_\phi$, we need some
identities. We begin with
$$
\begin{aligned}
[(\xi_1\ro\eta_1)\ro(\xi_2\ro\eta_2)](\zeta\ro\omega)
&=(\zeta\ro\omega|\xi_2\ro\eta_2)\xi_1\ro\eta_1\\
&=(\zeta|\xi_2)(\eta_2|\omega)\xi_1\ro\eta_1
=(\xi_1\ro\xi_2)(\zeta\ro\omega)(\eta_2\ro\eta_1),
\end{aligned}
$$
which implies the identity
\begin{equation}\label{hhhh}
(\xi_1\ro\eta_1)\ro(\xi_2\ro\eta_2)=R_{\eta_2\ro\eta_1}L_{\xi_1\ro\xi_2}\in \calb(H),
\end{equation}
for vectors $\xi_1\ro\eta_1$ and $\xi_2\ro\eta_2$ in $H$.
We also have the following identity
\begin{equation}\label{idddd}
\lan x\ro y, \omega_{\xi,\eta}\cdot \sigma \ran=\lan x\xi\ro y\eta, \sigma \ran,\qquad
x,y\in H,
\end{equation}
for $\sigma\in\calb(K)_*$ and $\omega_{\xi,\eta}\in\calb(K)^\pr_*$ given by $\omega_{\xi,\eta}(R_y)=(y\xi|\eta)$.
To see this, we write
$x=\sum_ix_i e_i\ro f_i$ and $y=\sum_j y_j g_j\ro h_j$ in $\calhs(K)$. By (\ref{hhhh}), we have
$x\ro y
=\sum_{i,j} x_i{\bar y}_j  R_{h_j\ro f_i} L_{e_i\ro g_j}$. By the identity
$$
\lan h_j\ro f_i, \omega_{\xi,\eta} \ran=((h_j\ro f_i)\xi|\eta)=(\xi|f_i)(h_j|\eta),
$$
we have
$$
\begin{aligned}
\lan x\ro y, \omega_{\xi,\eta}\cdot\sigma \ran
&=\textstyle\sum_{i,j} x_i{\bar y}_j\lan h_j\ro f_i, \omega_{\xi,\eta} \ran \lan e_i\ro g_j, \sigma \ran\\
&=\textstyle\sum_{i,j} x_i{\bar y}_j (\xi|f_i)(h_j|\eta) \lan e_i\ro g_j, \sigma \ran\\
&=\lan (\textstyle\sum_i x_i (\xi|f_i)e_i)\ro (\textstyle\sum_j y_j(\eta|h_j)g_j), \sigma \ran\\
&=\lan (\textstyle\sum_i x_i(e_i\ro f_i)\xi)\ro(\textstyle\sum_j y_j(g_j\ro h_j)\eta), \sigma \ran,
\end{aligned}
$$
as was required.

\begin{theorem}
Suppose that $\calb(K)$ acts on the Hilbert space $H=\calhs(K)$ by left multiplications.
For $\phi\in\calcb^\sigma(\calb(K))$, the following are equivalent:
\begin{enumerate}
\item[{\rm (i)}]
$\phi$ is positive;
\item[{\rm (ii)}]
$\lan\choi_\phi,\sigma^\pr\cdot\tau\ran\ge 0$ for every $\sigma^\pr\in\calb(K)^{\pr +}_*$ and $ \tau \in\calb(K)_*^+$.
\end{enumerate}
\end{theorem}

\begin{proof}
For every $\sigma^\pr\in\calm^{\pr {+}}_*$ and $\tau\in{\calm^+_*}$, we have the identity
$$
\lan \choi_\phi, \sigma^\pr\cdot\tau\ran
=\lan E_0, \tilde\phi_*(\sigma^\pr\cdot\tau)\ran
=\lan E_0, \sigma^\pr\cdot \phi_*(\tau)\ran,
$$
by (\ref{iddddxxx}), and so we see that the positivity of $\phi$ implies (ii).
For $\omega_{\xi,\xi}\in\calb(K)^{\pr +}_*$, $\tau\in \calb(K)^+_*$ and $\xi\in K$, we also have
$$
\begin{aligned}
\lan\choi_\phi,\omega_{\xi,\xi}\cdot\tau\ran
&=\lan x_0\ro x_0, \tilde\phi_*(\omega_{\xi,\xi}\cdot\tau)\ran\\
&=\lan x_0\ro x_0, \omega_{\xi,\xi}\cdot\phi_*(\tau)\ran
=\lan x_0\xi\ro x_0\xi, \phi_*(\tau) \ran,
\end{aligned}
$$
by (\ref{iddddxxx}) and (\ref{idddd}). Note that $\{x_0\xi:\xi\in K\}$ is dense in $K$ by Proposition \ref{dense-vec},
so that (ii) implies (i).
\end{proof}

Suppose that $K$ is a finite dimensional Hilbert space with the usual orthonormal basis $\{e_1,\dots, e_n\}$. Then
$\id_K=\sum_{i=1}^n e_i\ro e_i\in H=\calhs(K)$ is a separating and cyclic vector for $L_{\calb(K)}$ acting on
$H$ by left multiplications. We have
$$
\id_{K}\ro\id_{K}
=\sum_{i,j=1}^n(e_i\ro e_i)\ro (e_j\ro e_j)
=\sum_{i,j=1}^n R_{e_j\ro e_i}L_{e_i\ro e_j},
$$
by (\ref{hhhh}), and so we also have
$$
\choi_\phi=\tilde\phi(\id_{K}\ro\id_{K})=\sum_{i,j=1}^n R_{e_j\ro e_i}L_{\phi(e_i\ro e_j)}
=\sum_{i,j=1}^n R_{(e_i\ro e_j)^*}L_{\phi(e_i\ro e_j)},
$$
up to a scalar multiplication.
In the diagram (\ref{tensor-left}), this corresponds to $\sum_{i,j=1}^n e_{ij}\ot\phi(e_{ij})$,
which is the usual Choi matrix of $\phi$.

\begin{remark}\label{rem-finite}
For given $x\in{\calb(K)}$ and $y\in{\calhs(K)}$, we have $xy\in\calhs(K)$ and
$$
xy\ro xy =L_xy\ro L_xy
=L_x(y\ro y)(L_x)^*
=\ad_{L_{x^*}}(y\ro y) = \widetilde{\ad_{x^*}}(y\ro y),
$$
as operators on $H=\calhs(K)$.
The last identity follows from
$$
\ad_{L_{x^*}}(R_y L_z) = L_x R_y L_z L_{x^*} = R_y L_x L_z L_{x^*} = R_y L_{xzx^*} = \widetilde{\ad_{x^*}}(R_y L_z)
$$
for $x,y,z \in {\calb(K)}$.
When $K$ is finite dimensional, we see that
$$
\choi_\phi=\tilde\phi(x_0\ro x_0)
=\tilde\phi\circ\widetilde{\ad_{x^*_0}}(\id_K\ro\id_K)
$$
is the usual Choi matrix of the map $\phi\circ\ad_{x_0^*}:K\to K$ up to a scalar multiplication,
where $\id_K\in\calhs(K)$ is considered as a vector in $H=\calhs(K)$. Because $x_0$ is an invertible operator on $K$,
we see by \cite[Theorem IV.2]{han_kye_Choi_mat} that $\phi\mapsto\choi_\phi$ retains the correspondence
between $k$-superpositivity (respectively $k$-positivity) and Schmidt number $k$ (respectively $k$-block-positivity)
when $K$ is finite dimensional. See \cite{{kye_comp-ten},{kye_lec_note}} for surveys on these topics.
\end{remark}

\section{Schmidt numbers and $k$-positive maps}\label{sec_d}

Throughout this section, we suppose that $\calm=\calb(K)$ is the type ${\rm I}_\infty$ factor acting on
$H=\calhs(K)$ by left multiplications as in the last section.
Following Proposition 1 of \cite{Shirokov_2013}, we denote by $\cals_k$ the norm closed convex cone
of $\calt(H)$ generated by $\omega_{x,x}$ for $x\in H=\calhs(K)$ with $\rank x\le k$, and
we say that a normal positive functional
$\varrho\in\calt(H)$ has {\sl  Schmidt number $k$} if it belongs to $\cals_k\setminus\cals_{k-1}$.
A normal state $\varrho$ is called separable if $\varrho\in\cals_1$.
See \cite{{holevo_sep},{stormer_2008}} for $k=1$.
We note that the notion of Schmidt rank for pure states on an infinite dimensional Hilbert space
was studied extensively in \cite{vLSSWerner}.
It was shown in \cite{eom-kye} that a linear map $\phi:M_m\to M_n$ is $k$-positive if and
only if $\lan \choi_\phi, zz^* \ran\ge 0$ for every column vector $z\in\mathbb C^m\ot\mathbb C^n$
with Schmidt rank $\le k$. We begin with the infinite dimensional analogue.

\begin{theorem}\label{k-pos}
Suppose that $\calb(K)$ is acting on $H=\calhs(K)$ by left multiplications.
For a linear map $\phi\in\calcb^\sigma(\calb(K))$, the following are equivalent:
\begin{enumerate}
\item[{\rm (i)}]
$\phi$ is $k$-positive;
\item[{\rm (ii)}]
$\lan\choi_\phi, \omega_{x,x} \ran \ge 0$ for every $x\in H=\calhs(K)$ with $\rk x\le k$;
\item[{\rm (iii)}]
$\lan x\ro x,\choid_\phi\ran\ge 0$ for every $x\in H=\calhs(K)$ with $\rk x\le k$.
\end{enumerate}
\end{theorem}

\begin{proof}
We calculate $\lan \choi_\phi, \omega_{x,x}\ran$ for a given $x=\sum_{i=1}^k \xi_i\ro \eta_i\in H=\calhs(K)$.
We begin with
$$
\begin{aligned}
\lan R_z L_w, \omega_{x,x} \ran
&=(wxz|x)\\
&=\textstyle\sum_{i,j=1}^k(w\xi_i\ro z^* \eta_i|\xi_j\ro \eta_j)\\
&=\textstyle\sum_{i,j=1}^k(w\xi_i|\xi_j)(\eta_j|z^* \eta_i)
=\textstyle\sum_{i,j=1}^k \lan z, \omega_{\eta_j,\eta_i} \ran \lan w, \omega_{\xi_i,\xi_j} \ran,
\end{aligned}
$$
for $z,w\in\calb(K)$, which implies
$\omega_{x,x}=\sum_{i,j=1}^k \omega_{\eta_j,\eta_i}\cdot \omega_{\xi_i,\xi_j}$
with $\omega_{\eta_j,\eta_i}\in\calb(K)_*$ and $\omega_{\xi_i,\xi_j} \in \calb(K)_*$.
Therefore, we have
$$
\begin{aligned}
\lan \choi_\phi, \omega_{x,x} \ran
&=\lan x_0\ro x_0, \tilde\phi_*(\omega_{x,x}) \ran\\
&=\textstyle\sum_{i,j=1}^k \lan x_0\ro x_0, \omega_{\eta_j,\eta_i}\cdot \phi_*(\omega_{\xi_i,\xi_j}) \ran\\
&=\textstyle\sum_{i,j=1}^k \lan x_0\eta_j\ro x_0\eta_i, \phi_*(\omega_{\xi_i,\xi_j}) \ran
=\textstyle\sum_{i,j=1}^k (\phi(x_0\eta_j\ro x_0\eta_i)\xi_i|\xi_j),
\end{aligned}
$$
by (\ref{idddd}). In conclusion, we have
$\lan \choi_\phi, \omega_{x,x} \ran
=((\id_k\ot\phi)(\eta\ro \eta)\xi|\xi)$,
with $\xi=\sum_{i=1}^k e_i\ot\xi_i$ and $\eta=\sum_{i=1}^k e_i\ot x_0\eta_i$ in $\mathbb C^k\ot K$.
Therefore, we conclude that (i) and (ii) are equivalent, since $\{x_0\eta:\eta\in K\}$ is dense in $K$.

We have $x\ro x={\sum_{i,j=1}^k} R_{\eta_j\ro\eta_i}L_{\xi_i\ro\xi_j}$ by (\ref{hhhh}), and so
$\tilde\phi(x\ro x)={\sum_{i,j=1}^k} R_{\eta_j\ro\eta_i}L_{\phi(\xi_i\ro\xi_j)}$. Therefore,
we also have
$$
\begin{aligned}
\lan x\ro x,\choid_\phi\ran
&=\textstyle\sum_{i,j=1}^k ( R_{\eta_j\ro\eta_i}L_{\phi(\xi_i\ro\xi_j)} x_0|x_0)\\
&=\textstyle\sum_{i,j=1}^k (\phi(\xi_i\ro\xi_j)x_0|x_0(\eta_i\ro\eta_j))
=\textstyle\sum_{i,j=1}^k(\phi(\xi_i\ro\xi_j)x_0\eta_j|x_0\eta_i)
\end{aligned}
$$
by (\ref{id-1}) and (\ref{id-2}), and the equivalence (i) $\Longleftrightarrow$ (iii) follows as before.
\end{proof}

See \cite{hlpqs} for another characterization of $k$-positive maps on infinite dimensional algebra $\calb(K)$.
For a subset $S\subset\calb(H)$, we denote by $S_\circ$ the convex cone of all $\varrho\in\calt(H)$ satisfying
$\lan X, \varrho \ran\ge 0$ for every $X\in S$. For a subset $S\subset\calt(H)$, we also denote by
$S^\circ$ the convex cone of all $X\in\calb(H)$ such that $\lan X, \varrho \ran\ge 0$ for every $\varrho\in S$.
For $S\subset\calt(H)$, we know that $(S^\circ)_\circ$ is the smallest closed convex cone containing $S$.
Therefore, we have $\cals_k=(\cala_k^\circ)_\circ$, with the set $\cala_k$ of all $\omega_{x,x}\in\calt(H)$ with
$\rk x\le k$ as an operator on $K$. We denote by $\mathbb P^\sigma_k$ the convex cone of all
$\phi\in\calcb^\sigma(\calb(K))$ which are $k$-positive,
and put $\choi_{\mathbb P^\sigma_k}=\{\choi_\phi:\phi\in \mathbb P^\sigma_k\}$.
Then Theorem \ref{k-pos} tells us that $\choi_{\mathbb P^\sigma_k}\subset \cala_k^\circ$.

\begin{lemma}
$\choi_{\mathbb P^\sigma_k}$ is weak${}^*$ dense in $\cala_k^\circ$.
\end{lemma}

\begin{proof}
For a given finite dimensional subspace $F$ of $K$, we put $H_F=\calhs(F,x_0F)$.
We note that the separating and cyclic vector $x_0$ is an  injective operator on $K$, and so
the Hilbert spaces $F$ and $x_0F$ share the same dimension.
We define $\iota_F:H_F\to H$ and $\pi_F:H\to H_F$ by
$$
\iota_F(y):K\to F\xrightarrow{y} x_0F\hookrightarrow K,
\qquad
\pi_F(x): F\hookrightarrow K\xrightarrow{x} K\to x_0F,
$$
for $y\in H_F$ and $x\in H$, with projections $K\to F$ and $K\to x_0F$.
We note that $\calb(x_0F)$ acts on the Hilbert space $H_F$ by left multiplications
with a separating and cyclic vector $\pi_F(x_0)$.
For $X:H\to H$, we also define $X_F:H_F\to H_F$
by $X_F=\pi_F\circ X\circ\iota_F$.

Now, we suppose that $X\in\calb(H)$ belongs to $\cala_k^\circ$, then $X_F$ also belongs to the convex cone
$\{\omega_{x,x}\in\calt(H_F):\rk x\le k\}^\circ$. Indeed, we have
$$
\lan X_F, \omega_{x,x} \ran
= (\pi_F \circ X \circ \iota_F (x) | x)_{H_F}
= (X \circ \iota_F (x) | \iota_F (x))_H
= \lan X, \omega_{\iota_F(x),\iota_F(x)} \ran
\ge 0
$$
for $x \in H_F$ with $\rk x \le k$,
because $\iota_F(x) \in H$ satisfies $\rk\iota_F(x) \le k$.
Therefore, we have
$X_F=\choi_{\psi_F}$ for a $k$-positive map $\psi_F:{\calb(x_0F)\to \calb(x_0F)}$,
by Remark \ref{rem-finite}.
We denote by {$\phi_F:\calb(K)\to \calb(x_0F)\xrightarrow{\psi_F} \calb(x_0F)\hookrightarrow \calb(K)$.
Then one may check
$$
\pi_F \circ \tilde \phi_F(R_y L_x) \circ \iota_F = \tilde \psi_F (\pi_F \circ R_y L_x \circ \iota_F), \qquad x,y \in \calb(K),
$$
by calculating through $2 \times 2$ block matrices,
to see that the following diagram
$$
\begin{CD}
\calb(H) @         >\widetilde{\phi_F}>>      \calb(H) \\
      @V  VV                   @VV  V \\
\calb(H_F)  @         >\widetilde{\psi_F}>>       \calb(H_F)
\end{CD}
$$
commutes, where the vertical arrow is given by $X\mapsto X_F$ which sends
$E_0=x_0\ro x_0$ to $\pi_F(x_0)\ro\pi_F(x_0)$.} Therefore, we see that $\choi_{\phi_F}\in\calb(H)$ coincides on $H_F$ with
the Choi matrix $\choi_{\psi_F}\in \calb(H_F)$ with respect
to the separating and cyclic vector $\pi_F(x_0)\in H_F$.
Now, we consider the net of all finite dimensional subspaces of $H$ with respect to inclusion.
Then we see that $X$ is the weak${}^*$ limit of $\choi_{\phi_F}$ with $\phi_F\in \mathbb P_k^\sigma$,
since $\{\phi_F\}$ is uniformly bounded.
\end{proof}

Therefore, we have $(\choi_{\mathbb P^\sigma_k})_\circ=((\cala_k)^\circ)_\circ=\overline{\cala_k}= \cals_k$,
and get the following:

\begin{theorem}\label{theorem---sch}
Suppose that $\calm=\calb(K)$ is acting on $H=\calhs(K)$ by left multiplications. For a  normal positive
functional $\varrho$ on $\calb(H)$, the following are equivalent:
\begin{enumerate}
\item[{\rm (i)}]
$\varrho$ has Schmidt number not bigger than $k$;
\item[{\rm (ii)}]
$\lan \choi_\phi, \varrho \ran\ge 0$ for every $\phi\in\mathbb P_k^\sigma$.
\end{enumerate}
\end{theorem}

As an easy corollary, we get the criteria \cite{Shirokov_2013} for Schmidt numbers
in terms of the ampliation of $k$-positive maps,
which is an infinite dimensional analogue of \cite{terhal-schmidt}.
When $k=1$, this is the separability criteria by Horodecki's \cite{horo-1}.
See also \cite{{stormer_2008},{hou_2010}} for infinite dimensional case with $k=1$.

\begin{corollary}\label{th-schmidt}
A normal positive functional $\varrho$ on $\calb(H)$ has Schmidt number not bigger than $k$
if and only if
$\tilde\phi_*(\varrho)$ is positive for every $\phi\in\mathbb P_k^\sigma$.
\end{corollary}

\begin{proof}
We see that the identity
$$
\lan \choi_{\phi\circ\ad_V}, \varrho \ran
=\lan \tilde \phi \circ \widetilde{\ad_V} (E_0), \varrho \ran
=\lan \tilde \phi \circ \ad_{L_V} (E_0), \varrho \ran
=\lan V^*x_0 \ro V^*x_0, \tilde\phi_*(\varrho) \ran,
$$
holds for every $V\in\calb(K)$ by Remark \ref{rem-finite}.
Because $x_0$ is cyclic for $\calb(K)$, we get the conclusion.
\end{proof}

Recall that every normal positive functional is of the form $\varrho=\sum_{i=1}^\infty\lambda_i\omega_{x_i,x_i}$.
It is clear that $\varrho\in\cals_k$ if $\rk x_i\le k$ for each $i=1,2,\dots$ in this expression,
but such $\varrho$\rq s in $\cals_k$ do not exhaust all normal functionals in $\cals_k$.
See \cite{{holevo_sep},{Shirokov_2013}}.

\section{$k$-superpositive maps}

We say that $\phi\in\calcb^\sigma(\calb(K))$ is {\sl $k$-partially entanglement breaking} \cite{{cw-EB},{Shirokov_2013}}
if $\tilde\phi_*(\varrho)\in\cals_k$
for every normal state $\varrho\in\calt(H)$. Theorem \ref{entbreak} tells us that $\phi$
is $k$-partially entanglement breaking if and only if $\choid_\phi\in\cals_k$, and we see that the condition
$\choid_\phi\in\cals_k$ does not depend on the choice of a separating and cyclic vector.
By the identity $\choid_{\ad_V}=\omega_{Vx_0,Vx_0}$, we see that $\choid_{\ad_V}\in\cals_k$ whenever $\rk V\le k$.
We denote by $\mathbb E_k$ the convex cone of $\calcb^\sigma(\calb(K))$ generated by $\{\ad_V:\rk V\le k\}$,
then we see that every map $\phi\in \mathbb E_k$ is $k$-partially entanglement breaking.

For a given $\phi\in\calcb^\sigma(\calb(K))$, we define $\|\phi\|_\choid=\| \choid_\phi\|_{\calt(H)}$. Then
$\phi\mapsto \|\phi\|_\choid$ is a norm on the space $\calcb^\sigma(\calb(K))$ by Theorem \ref{inj}.
We say that a normal completely positive map from $\calb(K)$ into itself
is {\sl $k$-superpositive} \cite{ssz} when it belongs to
the closure of $\mathbb E_k$ with respect to the norm $\|\cdot\|_\choid$.
We note that the norm $\|\cdot \|_\choid$ itself depends on the choice of a separating and cyclic vector $x_0$,
but the next proposition tells us that the definition of $k$-superpositivity does not depend on
the choice of a separating and cyclic vector.

\begin{proposition}\label{k-super-choi}
A normal completely positive map $\phi:\calb(K)\to\calb(K)$  is
$k$-super-positive if and only if $\choid_\phi\in\cals_k$.
\end{proposition}

\begin{proof}
If $\phi_i\in\mathbb E_k$ and $\|\phi_i-\phi\|_\choid\to 0$ then
we have $\choid_{\phi_i}\in\cals_k$ and $\|\choid_{\phi_i}-\choid_\phi\|_{\calt(H)}\to 0$, and so we have $\choid_\phi\in\cals_k$.
For the converse, we suppose that $\choid_\phi\in\cals_k$, and $\e>0$ is given. Then we take $x_1,x_2,\dots,x_n\in H=\calhs(K)$
such that $\rk x_i\le k$ and $\|\choid_\phi-\sum_{i=1}^n\omega_{x_i,x_i}\|_{\calt(H)}<\e$.
By Proposition \ref{dense-vec} (ii), we can
take $V_i\in\calb(K)$, $i=1,2,\dots,n$, such that $\rk V_i\le k$ and
$$
\|V_ix_0-x_i\|\le \e/n(2M+\e),\qquad i=1,2,\dots,n,
$$
where $M$ is the maximum of $\{\|x_i\|:i=1,2,\dots,n\}$.

For given $x,y\in H$, we have
$$
\begin{aligned}
\|\omega_{x,x}-\omega_{y,y}\|_{\calt(H)}
&\le \|\omega_{x,x}-\omega_{x,y}\|_{\calt(H)}+\|\omega_{x,y}-\omega_{y,y}\|_{\calt(H)}\\
&\le \|x\|\|x-y\|+\|x-y\|\|y\|
=(\|x\|+\|y\|)\|x-y\|.
\end{aligned}
$$
Therefore, we have
$$
\begin{aligned}
\|\phi-\textstyle\sum_{i=1}^n \ad_{V_i}\|_\choid
&=\|\choid_\phi-\textstyle\sum_{i=1}^n \omega_{V_ix_0,V_ix_0}\|_{\calt(H)}\\
&\le\|\choid_\phi-\textstyle\sum_{i=1}^n\omega_{x_i,x_i}\|_{\calt(H)}
   +\textstyle\sum_{i=1}^n\|\omega_{x_i,x_i}-\omega_{V_ix_0,V_ix_0}\|_{\calt(H)}\\
&<\e  +\textstyle\sum_{i=1}^n (\|x_i\|+\|V_ix_0\|)\|x_i-V_ix_0\|<2\e.
\end{aligned}
$$
Hence, we conclude that $\phi$ belongs to the closure of $\mathbb E_k$, and it is $k$-superpositive.
\end{proof}

The behaviour of $\{\|\phi_n\|_\choid\}$ for a sequence $\{\phi_n\}$ heavily depends on the choice of the separating and cyclic vector $x_0$,
even though the notion of $k$-superpositivity does not depend on the choice of $x_0$.
To see this, we take a separating and cyclic vector
$$
x_0=\sum_{n=1}^\infty n^{-(1+p)/2} e_n\ro e_n\in H=\calhs(K),
$$
with an orthonormal basis $\{e_n\}$ of $K$, which depends on the choice of $p>0$.
We also take $V_n=n^{(1+\e)/2} e_n\ro e_n\in\calb(K)$, and put $\phi_n=\ad_{V_n}$.
Then we have $\choid_{\phi_n}=V_nx_0\ro V_nx_0$, and so we have
$\|\phi_n\|_\choid=\|V_nx_0\|^2_H=n^{\e-p}$. Therefore, $\lim_{n\to\infty}\|\phi_n\|_\choid=0,1$ or $+\infty$ depending on the choice of $p>0$.
We note that $\{\phi_n\}$ is neither uniformly bounded nor converges to zero in the point-weak${}^*$ topology.

Suppose that $\{\phi_n\}$ is uniformly bounded sequence in $\calcb^\sigma(\calm)$ and $\|\phi_n-\phi\|_\choid\to 0$.
Then the identity
\begin{equation}\label{bdjgnkvfv}
\lan A^\pr B C^\pr, \choid_\phi \ran
=( \phi(B) C^\pr x_0|A^{\pr *}x_0),
\qquad A^\pr, C^\pr\in\calm^\pr,\ B\in\calm
\end{equation}
shows that $(\phi_n(B)\xi|\eta)\to (\phi(B)\xi|\eta)$ for $\xi,\eta$ in the dense subset $\calm x_0$ of $H$.
Using the uniform boundedness of $\{\phi_n\}$, we see that $\phi_n(B)\to \phi(B)$ in the weak operator topology,
or equivalently weak${}^*$ topology, for every $B\in\calb(K)$. This proves (i) $\Longrightarrow$ (iii)
of the following theorem for a sequence $\{\phi_n\}$ in $\calcb^\sigma(\calm)$
with an arbitrary factor $\calm$ acting on a separable Hilbert space with a separating and cyclic vector.

\begin{theorem}
Suppose that $\phi_n$ and $\phi$ are normal completely positive maps on $\calb(K)$. Then the following are equivalent:
\begin{enumerate}
\item[{\rm (i)}]
$\{\phi_n\}$ is uniformly bounded and $\|\phi_n-\phi\|_\choid\to 0$;
\item[{\rm (ii)}]
$\{\phi_n\}$ is uniformly bounded, $\|\choid_{\phi_n}\|_{\calt(H)}\to\|\choid_\phi\|_{\calt(H)}$
and $\choid_{\phi_n}\to\choid_\phi$ weakly;
\item[{\rm (iii)}]
$\phi_n\to\phi$ in the point weak${}^*$-topology.
\end{enumerate}
\end{theorem}

\begin{proof}
We note that $\choid_{\phi_n}$ and $\choid_\phi$ are normal positive functionals by Theorem \ref{thm}.
It is known \cite{dellantonio} that if a sequence of normal states converges weakly then it converges in norm,
and so we see that  (i) is equivalent to (ii).
If a sequence $\{\phi_n(I)\}$  converges in the weak${}^*$ topology,
then $\{\phi_n(I)\}$ is uniformly bounded by the uniform boundedness principle.
Since $\phi_n$ is completely positive, we have $\|\phi_n\|=\|\phi_n(I)\|$,
and so we see that (iii) implies that $\{\phi_n\}$ is uniformly bounded.
We also have $\|\choid_{\phi_n}\|_{\calt(H)}\to\|\choid_\phi\|_{\calt(H)}$, by the identity
$$
\|\choid_\phi\|_{\calt(H)}
=\tr (\choid_\phi)
=\lan I, \choid_\phi \ran
=(\phi(I)x_0|x_0).
$$
Therefore, it remains to show that (iii) implies that
$\choid_{\phi_n}\to \choid_\phi$ weakly.

Now, we suppose that (iii) holds.
For every finite dimensional subspace $F$ of $K$, we denote by $\pi(F)\in\calb(K)$ the projection onto $F$.
Since $R_{\pi(F)}R_yL_xR_{\pi(F)}=R_{\pi(F)y\pi(F)}L_x$
and $\pi(F)y\pi(F)\in\calb(F)$,
we see that $R_{\pi(F)}XR_{\pi(F)}$ belongs to the algebra generated
by $L_{\calb(K)}$ and its commutant $R_{\calb(K)}$ for every $X\in\calb(H)$. Then the identity (\ref{bdjgnkvfv}) tells us that
$$
\lan R_{\pi(F)}XR_{\pi(F)}, \choid_{\phi_n}-\choid_\phi \ran \to 0,\qquad {\text{\rm as}}\ n\to\infty,
$$
for every $X\in\calb(H)$.
We are going to show $\lan X, \choid_{\phi_n}-\choid_\phi \ran \to 0$ for every $X\in\calb(H)$.
Now, we have
\begin{equation}\label{nvhmbgjj}
\begin{aligned}
X
&=R_{\pi(F)}XR_{\pi(F)}
+R_{\pi(F)}X(I_H-R_{\pi(F)})\\
&\phantom{XXXX}+(I_H-R_{\pi(F)})XR_{\pi(F)}
+ (I-R_{\pi(F)})X(I_H-R_{\pi(F)}).
\end{aligned}
\end{equation}
As for the second term of (\ref{nvhmbgjj}), the pairing $\lan R_{\pi(F)}X(I_H-R_{\pi(F)}), \choid_{\phi_n}-\choid_\phi \ran$
with $\choid_{\phi_n}-\choid_\phi$ is equal to
$$
\lan R_{\pi(F)}(\tilde\phi_n(X)-\tilde\phi(X))(I_H-R_{\pi(F)}), E_0 \ran
=\lan \tilde\phi_n(X)-\tilde\phi(X), (I_H-R_{\pi(F)})E_0 R_{\pi(F)} \ran
$$
by the $R_{\calb(K)}$-bimodule property of $\tilde\phi$. Putting $M=\sup\|\phi_n\|$, we have
$$
|\lan R_{\pi(F)}X(I_H-R_{\pi(F)}), \choid_{\phi_n}-\choid_\phi \ran|
\le 2M\|x_0\|\|x_0-x_0\pi(F)\|\|X\|,
$$
which becomes arbitrarily small if we take $F$ with the sufficiently large dimension.
The third and fourth terms in (\ref{nvhmbgjj}) can be controlled in the exactly same way. Therefore, we conclude that
$\lan X, \choid_{\phi_n}-\choid_\phi \ran$ tends to zero for every $X\in\calb(H)$.
\end{proof}

It would be nice to know whether every $k$-superpositive map is the limit of a
\lq uniformly bounded\rq\ sequence in $\mathbb E_k$
with respect to $\|\ \|_\choid$.
We note that the condition $\choid_\phi\in\cals_k$ in Proposition \ref{k-super-choi} holds
if and only if $\lan \choi_\psi, \choid_\phi \ran \ge 0$ for every $\psi\in\mathbb P_k^\sigma$ by
Theorem \ref{theorem---sch},
and several equivalent conditions for this have been considered in
Theorem \ref{horo}.
We also see that the condition $\choid_\phi\in\cals_k$ is equivalent to say that $\phi$ is $k$-partially
entanglement breaking by Theorem \ref{entbreak}.
We summarize as follows:

\begin{theorem}\label{k-superpos}
For a given normal completely positive map $\phi:\calb(K)\to\calb(K)$, the following are equivalent:
\begin{enumerate}
\item[{\rm (i)}]
$\phi$ is $k$-superpositive;
\item[{\rm (ii)}]
$\lan \choi_\psi, \choid_\phi \ran \ge 0$ for every $\psi\in\mathbb P_k^\sigma$;
\item[{\rm (iii)}]
{$\phi\circ\psi$} is completely positive for every $\psi\in\mathbb P_k^\sigma$;
\item[{\rm (iv)}]
{$\tilde\phi(\choi_\psi)$} is positive for every $\psi\in\mathbb P_k^\sigma$;
\item[{\rm (v)}]
{$\tilde\psi_*(\choid_\phi)$} is positive for every $\psi\in\mathbb P_k^\sigma$;
\item[{\rm (vi)}]
$\choid_\phi\in \cals_k$;
\item[{\rm (vii)}]
$\tilde\phi_*(\omega_{x,x})\in\cals_k$ for every $x\in H$;
\item[{\rm (viii)}]
$\phi$ is $k$-partially entanglement breaking.
\end{enumerate}
\end{theorem}

For the infinite dimensional case,
$k$-partially entanglement breaking maps with $k=1$ have been studied intensively
in \cite{holevo_sep} under the name of entanglement breaking maps,
for which the equivalence condition (iii) was given in \cite{he_2013}.
For general $k=1,2,\dots$, it was shown in \cite{Shirokov_2013} that
there exists a single pure state $\varrho$ such that $\id\ot\phi(\varrho)\in\cals_k$ implies that
$\phi$ is $k$-partially entanglement breaking. This is related to the condition (vi).

In the case of matrix algebras, a linear map whose Choi matrix is
separable has been called superpositive by Ando \cite{ando-04}, as
it was considered in the condition (vi) with $k=1$. On the other
hand, it was shown that conditions (i), (iii), (vi) and (viii) are
equivalent when $k=1$ in \cite{hsrus}, and maps satisfying those
condition were called entanglement breaking. It was also shown in
\cite{cw-EB} that the conditions (i), (iii), (vi) and (viii) are
equivalent for $k=1,2,\dots$. The equivalent condition like (ii)
using the bilinear pairing between maps was found in \cite{ssz},
where the authors call such maps $k$-superpositive maps.

In case of the $n\times n$ matrix algebra, there is a line segment in the
convex cone of all positive maps on which $k$-superpositivity and $k$-positivity are
distinguished for $k=1,2,\dots, n$ \cite{{choi72},{tomiyama-83},{tom_85},{terhal-schmidt}}.
See also \cite[Section 1.6, 1.7]{kye_lec_note}. Such maps can be extended on $\calb (K)$ with infinite dimensional $K$,
retaining $k$-superpositivity and $k$-positivity. It would be nice to find a line segment
in the mapping space $\calcb^\sigma(\calb(K))$ on which
$k$-superpositive maps and $k$-positive maps appear for every $k=1,2,\dots$.


\begin{thebibliography}{99}

\bibitem{ando-04}
T. Ando, \it Cones and norms in the tensor product of matrix spaces,
\rm Linear Alg. Appl. \bf 379 \rm (2004), 3--41.

\bibitem{blackadar}
B. Blackadar,
\lq\lq Operator Algebras: Theory of $C^*$-Algebras and von Neumann Algebra\rq\rq,
Springer-Verlag, 2006.

\bibitem{bolanos-quezada}
J. R. Bola\~ nos-Servin and R. Quezada,
\it The $\Theta$-KMS adjoint and time reversed quantum Markov semigroups,
\rm Infin. Dimens. Anal. Quantum Probab. Relat. Top. \bf 18 \rm (2015), 1550016.

\bibitem{choi72}
M.-D. Choi, \it Positive linear maps on $C^*$-algebras, \rm Canad.
Math. J. \bf 24 \rm (1972), 520--529.

\bibitem{choi75-10}
M.-D. Choi, \it Completely positive linear maps on complex matrices,
\rm Linear Alg. Appl. \bf 10 \rm (1975), 285--290.

\bibitem{cw-EB}
D. Chru\'{s}ci\'{n}ski and A. Kossakowski, \it On Partially
Entanglement Breaking Channels,
\rm Open Sys. Inform. Dynam. {\bf 13} (2006), 17--26.

\bibitem{dellantonio}
G. Dell\rq Antonio,
\it On the limits of sequences of normal states,
\rm Comm. Pure Appl. Math. {\bf 20} (1967), 413--429.

\bibitem{dePillis}
J. de Pillis, \it Linear transformations which preserve Hermitian
and positive semidefinite operators, \rm Pacific J. Math. {\bf 23}
(1967), 129--137.

\bibitem{duvenhage_snyman}
R. Duvenhage and M. Snyman,
\it Balance between quantum Markov semigroups,
\rm Ann. Henri Poincar\' e {\bf 19} (2018), 1747--1786.

\bibitem{effros_lance}
E. G. Effros and E. C. Lance,
\it Tensor products of operator algebras,
\rm Adv. Math. {\bf 25} (1977), 1--34.

\bibitem{effros-ruan}
E. G. Effros and Z.-J. Ruan,
On approximation properties for operator spaces,
Intern. J. Math. {\bf 1} (1990), 163--187.

\bibitem{eom-kye}
M.-H. Eom and S.-H. Kye, \it Duality for positive linear maps in
matrix algebras, \rm Math. Scand. \bf 86 \rm (2000), 130--142.

\bibitem{Friedland_2019}
S. Friedland,
\it Infinite dimensional generalizations of Choi's Theorem,
\rm Spec. Matrices {\bf 7} (2019); 67--77.

\bibitem{gks}
M. Girard, S.-H. Kye and E. St\o rmer, \it Convex cones in mapping
spaces between matrix algebras, \rm Linear Alg. Appl. {\bf 608}
(2021), 248--269.

\bibitem{grabowski_kus_marmo}
J. Grabowski, M. Ku\' s and G. Marmo,
\it On the relation between states and maps in infinite dimensions,
\rm Open Syst. Inf. Dyn. {\bf 14} (2007), 355--370.

\bibitem{guddar}
S. Gudder, \it Operator isomorphisms on Hilbert space tensor
products, \rm preprint. arXiv 2010.15901.

\bibitem{Haapasalo_2020}
E. Haapasalo,
\it The Choi--Jamio\l kowski isomorphism and covariant quantum channels,
\rm Quantum Stud.: Math. Found. {\bf 8} (2021), 351--373.

\bibitem{han_kye_tri}
K. H. Han and S.-H, Kye,
\it Various notions of positivity for bi-linear maps and applications to tri-partite entanglement,
\rm J. Math. Phys. {\bf 57} (2016), 015205.

\bibitem{han_kye_multi}
K. H. Han and S.-H. Kye,
\it Construction of multi-qubit optimal genuine entanglement witnesses,
\rm J. Phys. A: Math. Theor. {\bf 49} (2016), 175303.

\bibitem{han_kye_Choi_mat}
K. H. Han and S.-H. Kye,
\it Choi matrices revisited.\ II,
\rm J. Math. Phys. {\bf 64} (2023), 102202.

\bibitem{he_2013}
K. He,
\it On entanglement breaking channels for infinite
dimensional quantum systems,
\rm Intern. J. Theor. Phys. {\bf 52} (2013), 1886--1892.

\bibitem{holevo_2011}
A. S. Holevo,
\it Entropy gain and the Choi--Jamiolkowski correspondence for infinite dimensional quantum evolutions,
\rm Theor. Math. Phys. {\bf 166} (2011), 123--138.

\bibitem{holevo_2011_a}
A. S. Holevo,
\it The Choi--Jamiolkowski forms of quantum Gaussian channels,
\rm J. Math. Phys. {\bf 52} (2011), 042202.

\bibitem{holevo_sep}
A. S. Holevo, M. E. Shirokov and R. F. Werner,
\it Separability and entanglement-breaking in infinite dimensions,
\rm Russian Math. Surveys {\bf 60} (2005), arXiv:quant-ph/0504204v1.

\bibitem{horo-1}
M. Horodecki, P. Horodecki and R. Horodecki, \it Separability of
mixed states: necessary and sufficient conditions, \rm Phys. Lett. A
\bf 223 \rm (1996), 1--8.

\bibitem{hsrus}
M. Horodecki, P. W. Shor and M. B. Ruskai, \it Entanglement braking
channels, \rm Rev. Math. Phys. \bf 15 \rm (2003), 629--641.

\bibitem{hou_2010}
J. Hou,
\it A characterization of positive linear maps and
criteria of entanglement for quantum states,
\rm J. Phys. A: Math. Theor. {\bf 43} (2010), 385201.

\bibitem{hlpqs}
J. Hou, C.-K. Li, Y.-T. Poon, X. Qi and N.-S. Sze,
\it A new criterion and a special class of k-positive maps,
\rm Linear Alg. Appl. {\bf 470} (2015), 51--69.

\bibitem{jam_72}
A. Jamio\l kowski, \it Linear transformations which preserve trace
and positive semidefinite operators, \rm Rep. Math. Phys. {\bf 3}
(1972), 275--278.


\bibitem{KRI}
R. V. Kadison and J. R. Ringrose,
\lq\lq Fundamentals of the Theory of Operator Algebras, Vol. I,\rq\rq\
Academic Press, 1983.

\bibitem{kraus-71}
K. Kraus,
\it General state changes in quantum theory,
\rm Ann. of Phys. {\bf 64} (1971), 311--335.

\bibitem{kye_multi_dual}
S.-H. Kye, \it Three-qubit entanglement witnesses with the full
spanning properties, \rm J. Phys. A: Math. Theor. {\bf 48} (2015),
235303.

\bibitem{kye_Choi_matrix}
S.-H. Kye, \it Choi matrices revisited, \rm J. Math. Phys. {\bf 63}
(2022), 092202.

\bibitem{kye_comp-ten}
S.-H. Kye, \it Compositions and tensor products of linear maps
between matrix algebras, \rm Linear Algebra Appl. {\bf 658} (2023),
283--309.

\bibitem{kye_lec_note}
S.-H. Kye,
\lq\lq Positive Maps in Quantum Information Theory\rq\rq,
Lecture Notes, 2023, Seoul National Univ., {\tt http://www.math.snu.ac.kr/$\sim$kye/book/qit.html}

\bibitem{li_du}
Y. Li and H.-K. Du,
\it Interpolations of entanglement breaking channels and equivalent conditions for completely positive maps,
\rm J. Funct. Anal. {\bf 268} (2015), 3566--3599.

\bibitem{Magajna_2021}
B. Magajna,
\it Cones of completely bounded maps,
\rm Positivity {\bf 25} (2021), 1--29.

\bibitem{MvN}
F. J. Murray and J. von Neumann,
\it On Rings of Operators,
\rm Ann. of Math. {\bf 37} (1936), 116--229.

\bibitem{Paulsen_Shultz}
V. I. Paulsen and F. Shultz, \it Complete positivity of the map from
a basis to its dual basis, \rm J. Math. Phys. {\bf 54} (2013),
072201.

\bibitem{Shirokov_2013}
M. E. Shirokov,
\it Schmidt number and partially entanglement-breaking channels in infinite-dimensional quantum systems,
\rm Math. Notes {\bf 93} (2013), 766--779,
Original Russian Text published in Matematicheskie Zametki {\bf 93} (2013), 775--789.

\bibitem{ssz}
\L. Skowronek, E. St\o rmer, and K. \. Zyczkowski, \it Cones of
positive maps and their duality relations, \rm J. Math. Phys. {\bf
50}, (2009), 062106.

\bibitem{stormer}
E. St\o rmer, \it Positive linear maps of operator algebras, \rm
Acta Math. \bf 110 \rm (1963), 233--278.

\bibitem{stormer-dual}
E. St\o rmer, \it Extension of positive maps into $B(\mathcal H)$,
\rm J. Funct. Anal. \bf 66 \rm (1986), 235-254.

\bibitem{stormer_2008}
E. St\o rmer,
\it Separable states and positive maps,
\rm J. Funct. Anal. {\bf 254} (2008), 2303--2312.

\bibitem{stormer_2015}
E. St\o rmer,
\it The analogue of Choi matrices for a class of linear maps on Von Neumann algebras,
\rm Intern. J. Math. {\bf 26} (2015), 1550018.

\bibitem{tomiyama-83}
T. Takasaki and J. Tomiyama, \it On the geometry of positive maps in
matrix algebras, \rm Math. Z. \bf 184 \rm (1983), 101--108.

\bibitem{tomita_takesaki}
M. Takesaki,
\lq\lq Tomita's Theory of Modular Hilbert Algebra and Its Applications\rq\rq,
Lecture Notes Math. Vol. 128, Springer-Verlag, 1970.

\bibitem{terhal-schmidt}
B. M. Terhal and P. Horodecki, \it Schmidt number for density
matrices, \rm Phys. Rev. A \bf 61 \rm (2000), 040301.

\bibitem{tomczak}
N. Tomczak-Jaegermann,
\lq\lq Banach--Mazur Distances and Finite-Dimensional Operator Ideals\rq\rq,
Pitman Monog. Surv. Pure Appl. Math. Vol. 38,
Longman Scientific \& Technical, 1989.

\bibitem{tom59}
J. Tomiyama,
\it  On the projection of norm one in $W^*$-algebras. III,
\rm Tohoku Math. J. \rm (2) {\bf 11} (1959), 125--129.

\bibitem{tom_85}
J. Tomiyama, \it On the geometry of positive maps in matrix
algebras. II, \rm  Linear Alg. Appl.  \bf 69 \rm (1985), 169--177.

\bibitem{vLSSWerner}
L. van Luijk, R. Schwonnek, A. Stottmeister and R. F. Werner,
\it The Schmidt rank for the commuting operator framework,
\rm arXiv 2307.11619.

\bibitem{woronowicz}
S. L. Woronowicz, \it Positive maps of low dimensional matrix
algebras, \rm Rep. Math. Phys. \bf 10 \rm (1976), 165--183.

\end{thebibliography}
\end{document}